\documentclass[a4 paper,11pt]{amsart}

\usepackage{amsmath} %
\usepackage{amssymb} %
\usepackage{amscd} %
\usepackage{amsthm} %
\usepackage{mathrsfs} %
\usepackage{euscript}
\usepackage{eufrak}
\usepackage{ulem}
\usepackage{comment}
\usepackage[alphabetic]{amsrefs}
\usepackage{ascmac}

\usepackage[dvipdfmx]{graphicx} 

\usepackage{hyperref}
\usepackage{xcolor}
\definecolor{unbleu}{rgb}{0.03, 0.15, 0.4}

\hypersetup{
pdfborder = {0 0 0},
colorlinks,
linkcolor=unbleu,
citecolor=unbleu,
urlcolor=unbleu
}

\usepackage{fancyhdr}
 \newtheorem{theorem}{Theorem}[section]
 \newtheorem{lemma}[theorem]{Lemma}
 \newtheorem{proposition}[theorem]{Proposition}

\theoremstyle{definition}

\newtheorem{remark}[theorem]{Remark}
\newtheorem{example}[theorem]{Example}

\newcommand{\R}{\mathbb R}%     real number
\newcommand{\Z}{\mathbb Z}%      integer number
\def\Nset{\mathbb{N}}

\def\Rset{\mathbb{R}}

\def\M{\mathcal{M}}

\def\A{\mathcal{A}}

\def\Z{{Z}}
\def\W{{W}}
\def\B{\mathcal{B}}
\def\h{\mathfrak{h}}

\begin{document}

\title[]{Minimax aspects of optimizations in ergodic theory}

\author[S.Motonaga]{Shoya Motonaga}
\address{Department of Mathematical Sciences, Ritsumeikan University, 1-1-1 Noji-higashi, Kusatsu, Shiga 525-8577, Japan}
\email{motonaga@fc.ritsumei.ac.jp}

%\author[M. Shinoda]{Mao Shinoda}
%\address{Department of Mathematics, Ochanomizu University, 2-1-1 Otsuka, Bunkyo-ku, Tokyo, 112-8610, Japan}
%\email{shinoda.mao@ocha.ac.jp}

% \date{\today}

%%%%%%%%%%%%%%%%%%%%%%%%%%%%%%%%%%%%%%%%%%%%%%%%%%%%%%%%%
%%%%%%%%%%%%%%%%%%%%%%%%%%%%%%%%%%%%%%%%%%%%%%%%%%%%%%%%%
\begin{abstract}
We study optimization problems in ergodic theory from the view point of minimax problems. We give minimax characterizations of maximum ergodic averages involving time averages. Our approach also works for the abstract variational principle of generalized pressure functions which is proved by Bi\'{s} et al. (2022). We also describe the relationship between our minimax results and the Fenchel-Rockafellar duality.
\end{abstract}

\maketitle
% %\newpage

% \tableofcontents
%%%%%%%%%%%%%%%%%%%%%%%%%%%%%%%%%%%%%%%%%%%%%%%%%%%%%%%%
\section{Introduction}
In ergodic theory, there are several kinds of optimization problems. A typical example is the study of equilibrium measures, which are invariant probability measures that achieve the variational principle for topological pressure, and recently more general frameworks for equilibrium measures (and a variational principle for pressure functions) has been developed \cite{B22, B23}. Another example is ergodic optimization \cite{Jen06a}, in which one maximizes the integral of a given function among the set of invariant probability measures. It is known that the maximum ergodic average is equal to the upper limit of the maximum of the time average with respect to time, and the maximum of the upper limit of the time average (see Proposition~\ref{prop:jenkinson} in Section~\ref{sec:main} below). These optimizing problems deal with maximizations, but it is interesting to consider problems that can be formulated as minimax problems. In this paper, we reveal that these problems can actually be interpreted as minimax problems. In particular, we obtain minimax characterizations of the maximum ergodic averages for ergodic optimization, and give a simpler proof of the abstract variational principle for generalized pressure functions due to Bi\'{s} et al. \cite{B22, B23} by using of the minimax theorem. It is worth to notice that recently a minimax problem has also appeared in the study of mean dimension theory \cite{T20}. We hope that our results indicate further developments in minimax problems related to ergodic theory.

The rest of this paper is organized as follows. In Section~\ref{sec:pre}, we provide prerequisites on minimax theory. In Section~\ref{sec:main}, we consider minimax aspects of ergodic optimization and obtain some characterizations for the maximum ergodic average, and give a simpler proof of the abstract variational principle due to Bi\'{s} et al. \cite{B22, B23}. A connection between our minimax approach and the Fenchel-Rockafellar duality is clarified in Section~\ref{sec:FR}.

\section{Prerequisites: Minimax theory}\label{sec:pre}
In this section, we briefly review some basic facts of minimax problems in a minimum necessary form for our application.
Let $F$ be a real-valued function defined on the product set $\Z\times \W$ of two arbitrary sets $\Z,\W$ (not necessarily topologized). Then we easily obtain the following inequality.
\begin{proposition}[Minimax inequality]\label{prop:minimax_ineq}
	\begin{equation}\label{eqn:minimax_ineq}
		\vspace{1mm}
		\ \inf_{z\in \Z} \left(\sup_{w\in \W} F(z,w)\right) \ge \sup_{w\in \W} \left(\inf_{z\in\Z} F(z,w)\right)
	\end{equation}	
\end{proposition}
\begin{proof}
	Letting $\tilde{F}(z)=\sup_{w\in \W} F(z,w)$, we have
	\[
	\inf_{z\in\Z}\tilde{F}(z)\ge \inf_{z\in\Z} F(z,w'),\quad \forall w'\in\W.
	\]
	Taking the supremum on $w'\in\W$, we obtain the desired inequality.
\end{proof}

The main problem in minimax theory is when the equality in Eq.~\eqref{eqn:minimax_ineq} holds.
Von Neumann established the first progress of this kind in the study of game theory, and there are many 
generalization of his statement. Here we describe Fan's minimax theorem \cite{Fan53}. We recall the notion of convexity for a function on an arbitrary set.
The function $F$ is said to be convex on $\Z$, if for any two elements $z_1, z_2\in \Z$ and $\lambda\in [0,1]$ there exists $z_0\in\Z$ such that $f(z_0,w)\le \lambda f(z_1,w)+(1-\lambda)f(z_2,w)$ for all $w\in \W$.
Similarly, $F$ is said to be concave on $\W$, if for any two elements $w_1, w_2\in \W$ and $\lambda\in [0,1]$ there exists $w_0\in\W$ such that $f(z,w_0)\ge \lambda f(z,w_1)+(1-\lambda)f(z,w_2)$ for all $z\in \Z$.
\begin{theorem}[Minimax theorem \cite{Fan53}]\label{th:minmax}
	Let $Z$ be an arbitrary set (not topologized) and let $W$ be a compact Hausdorff space.
	Let $F$ be a real-valued function defined on the product set $\Z\times \W$ such that,
	for every $z\in Z$, $F(x,y)$ is upper semicontinuous on $W$.
	 If $F(z,w)$ is convex on $\Z$ 
		and concave on $\W$,
		then
		\begin{align}\label{eqn:minimax}
			\inf_{z\in \Z}\left(\sup_{w\in \W} F(z,w)\right)=\sup_{w\in \W}\left(\inf_{z\in \Z} F(z,w)\right).
		\end{align}
%
%	\begin{itemize}
%		\item[(i)]
%		\item[(ii)] Suppose that $Z$ is a compact Hausdorff space.
%		If $F(z,w)$ is lower semi-continuous and convex on $\Z$ for every $w\in \W$,
%		and concave on $\W$ for every $z\in \Z$,
%		then
%		\begin{align*}
%			\min_{z\in \Z}\sup_{w\in \W} F(z,w)=\sup_{w\in \W}\min_{z\in \Z} F(z,w).
%		\end{align*}
%		\item[(iii)] Suppose that $\Z,\W$ are compact Hausdorff spaces.
%		If $F(z,w)$ is lower semi-continuous and convex on $\Z$ for every $w\in \W$,
%		and upper semi-continuous and concave on $\W$ for every $z\in \Z$,
%		then
%		\begin{align*}
%			\min_{z\in \Z}\max_{w\in \W} F(z,w)=\max_{w\in \W}\min_{z\in \Z} F(z,w).
%		\end{align*}
%	\end{itemize}
\end{theorem}
Note that in \cite{Fan53} more detailed cases to hold Eq.~\eqref{eqn:minimax} are considered. 
We remark that any cases in \cite{Fan53} cannot be directly applied to Theorem~\ref{th:EO} in Section~\ref{sec:main} below.
	%\newpage

	\section{Minimax aspects of ergodic optimization and the abstract variational principle}\label{sec:main}
	\subsection{Maximum ergodic averages}
	In this section, we obtain minimax characterizations of the maximum ergodic average.
	Let $(X,d)$ be a compact metric space and let $T:X\to X$ be a continuous self-map.
	Let $\M(X)$ denote the set of Borel probability measures on $X$.
	Then $\M(X)$ is a compact subset of $C^*(X)$ in the weak $*$-topology where $C(X)$ is the set of continuous functions on $X$ equipped with the supremum norm.
	We write $\M_{T}(X)$ for the set of $T$-invariant Borel probability measures on $X$.
	In the study of ergodic optimization, for $\phi\in C(X)$, the quantity
	\[
	\alpha(\phi)=\sup_{\mu\in\M_T(X)} \int \phi\  d\mu,
	\]
	called as the \textit{maximum ergodic average} of $\phi$, and its miximizing measures are investigated.
	See \cite{Jen06a, Jen17}  for the details of ergodic optimization.
	We recall a well-known characterization of the maximum ergodic average
	involving time average.
	\begin{proposition}[\cite{Jen06a}]\label{prop:jenkinson} For $\phi\in C(X)$, the maximum ergodic average $\alpha(\phi)$ satisfies
	\begin{align*}
		\alpha(\phi)
		=\sup_{x\in \mathrm{Reg}_{T,\phi}} \lim_{n\to \infty} \frac{S_n\phi(x)}{n}
		=\sup_{x\in X} \limsup_{n\to \infty} \frac{S_n\phi(x)}{n}
		=\limsup_{n\to \infty} \sup_{x\in X} \frac{S_n\phi(x)}{n},
	\end{align*}
	where $S_n\phi=\sum_{i=0}^{n-1} \phi\circ T^i$ and $\mathrm{Reg}_{T,\phi}=\{x\in X;\ \lim_{n\to \infty} (1/n)S_n\phi(x)\ \text{exists}\}$.
	\end{proposition}
	We now consider minimax aspects of the maximum ergodic average $\alpha(\phi)$.
	The following result implies that we can replace ``$\limsup$" in Proposition~\ref{prop:jenkinson}
	with ``$\inf$" and thus 
	the maximum ergodic average can be expressed as an optimal value of  a minimax (maxmin)
	 problem.
	\begin{theorem}\label{th:EO} For $\phi\in C(X)$, the maximum ergodic average $\alpha(\phi)$ satisfies
		\begin{align*}
			\alpha(\phi)
		=\sup_{x\in X} \inf_{n\in \Nset} \frac{S_n\phi(x)}{n}
		%=\lim_{n\to \infty} \sup_{x\in X} \frac{S_n\phi(x)}{n} 
		=\inf_{n\in \Nset} \sup_{x\in X} \frac{S_n\phi(x)}{n} 
		\end{align*}
	\end{theorem}
	For the proof of Theorem~\ref{th:EO}, we need the following lemma.
	\begin{lemma}\label{lem:key}
		For $\phi\in C(X)$, it holds that
		\[
			\sup_{x\in X} \inf_{n\in \Nset} \frac{S_n\phi(x)}{n}=\sup_{x\in X} \liminf_{n\to +\infty} \frac{S_n\phi(x)}{n}.
		\]
	\end{lemma}
	We will prove Lemma~\ref{lem:key} after the proof of Theorem~\ref{th:EO}.
	We remark that Theorem~\ref{th:EO} relies on Lemma~\ref{lem:key} and the ideas of the proof of Proposition~\ref{prop:jenkinson} with the minimax inequality.
	\begin{proof}[Proof of Theorem~\ref{th:EO}]
		We first see that
		\begin{align}\label{eqn:KB}
			\limsup_{n\to+\infty} \sup_{x\in X} \frac{S_n\phi(x)}{n}\le \sup_{\mu\in\M_T(X)} \int \phi\ d\mu
			%\quad(=\liminf_{n\to+\infty} \sup_{x\in X} \frac{1}{n}S_n\phi(x)).
		\end{align}
		as in the original proof of Proposition~\ref{prop:jenkinson}. In fact, letting $\mu_n:=\frac{1}{n}\sum_{i=0}^{n-1} \delta_{T^i (x_n)}$ where $\delta_{p}$ stands for the Dirac measure supported at a point $p\in X$ and $x_n$ is a maximizer of $S_n \phi(x)$, we obtain an accumulation point $\mu^*\in \M(X)$ of $\{\mu_n\}$ and actually it satisfies $\mu^*\in \M_T(X)$, which yields the above inequality.
		
		Next, we see that
		\begin{align}\label{eqn:EOminmax}
		\alpha(\phi) 
		=\inf_{n\in \Nset} \sup_{x\in X} \frac{S_n\phi(x)}{n}
		\end{align}
		follows from the minimax inequality.
		%Fix $n_0\in\Nset$. Let $\X=\{n\in\Nset; n\ge n_0\}$, $\Y=\M(X)$, and $F(n, \mu)=\langle \frac{S_n\phi}{n},\mu\rangle$.
		%Then we compute
		%\begin{align*}
		%	\sup_{\mu\in\M(X)} \inf_{n\ge n_0} \left\langle \frac{S_n\phi}{n},\mu\right\rangle
		%	&=\sup_{\mu\in \M(X)} \begin{cases} \langle \phi, \mu\rangle & (\mu\in \M_T(X))\\
		%			\left\langle \phi,\frac{1}{n} \sum_{i=0}^{n-1} T^i_* (\mu)\right\rangle & (\mu\notin \M_T(X))
		%			\end{cases}
		%\end{align*}
		%and
		%\begin{align*}
		%	\inf_{n\ge n_0} \sup_{\mu\in\M(X)}  \left\langle \frac{S_n\phi}{n},\mu\right\rangle
		%	=\inf_{n\ge n_0} \sup_{x\in X}  \frac{S_n\phi(x)}{n}
		%\end{align*}
		%By the minimax theorem \red{(convexity? applicable?)} and \red{Lemma~}, we obtain
		%\[
		%	\sup_{\mu\in\M_T(X)} \int f\ d\mu
		%	=\liminf_{n\to+\infty} \sup_{x\in X} \frac{1}{n}S_nf(x).
		%\]
		%Fix $n_0\in\Nset$. 
		Let $\Z=\Nset$, $\W=\M(X)$, and $F(n, \mu)=\int \frac{S_n \phi}{n} d\mu$.
		Trivially,
		\begin{align}\label{eqn:EO1}
			\sup_{\mu\in\M_T(X)} \inf_{n\in\Nset}\int \frac{S_n\phi}{n}  \ d\mu
			\le \sup_{\mu\in\M(X)} \inf_{n\in\Nset} \int \frac{S_n\phi}{n}  \ d\mu.
		\end{align}
		 The left hand side of \eqref{eqn:EO1} becomes $\alpha(\phi)$
		 since $\frac{1}{n} \int S_n\phi \ d\mu=\int \phi \ d\mu$ holds for each $\mu\in\M_T$.
		 Moreover, by the minimax inequality \eqref{eqn:minimax_ineq} (Proposition~\ref{prop:minimax_ineq}), we have
		\begin{align}\label{eqn:EO2}
		\sup_{\mu\in\M(X)}\inf_{n\in\Nset}\int \frac{S_n \phi}{n}\ d\mu
		 \le \inf_{n\in\Nset}\sup_{\mu\in\M(X)} \int \frac{S_n \phi}{n}\ d\mu
		 =\inf_{n\in\Nset}\sup_{x\in X} \frac{S_n \phi(x)}{n}.
		\end{align}
		Note that for any $\phi'\in C(X)$ a Dirac measure supported at a maximum point of $\phi'$ attains the maximum of $\sup_{\mu\in \M(X)} \int \phi' \ d\mu$.
		It is clear that
		\begin{align}\label{eqn:EO3}
			\inf_{n\in\Nset}\sup_{x\in X} \frac{S_n \phi(x)}{n}
			\le \liminf_{n\to \infty}\sup_{x\in X} \frac{S_n \phi(x)}{n}
			\le \limsup_{n\to \infty}\sup_{x\in X} \frac{S_n \phi(x)}{n}
		\end{align}
		holds, and thus all inequalities in \eqref{eqn:KB}, \eqref{eqn:EO1}, \eqref{eqn:EO2}, \eqref{eqn:EO3} become equalities.
		
		Now we will prove that
		\begin{align}\label{eqn:key}
			\alpha(\phi)=\sup_{x\in X} \inf_{n\in \Nset} \frac{S_n\phi(x)}{n}.
		\end{align}
		We see that the proof of of Proposition~\ref{prop:jenkinson} actually implies that
		\begin{align}\label{eqn:key2}
			\alpha(\phi)=\sup_{x\in X} \liminf_{n\to+\infty} \frac{S_n\phi(x)}{n}
			 =\liminf_{n\to\infty}  \frac{S_n\phi(y)}{n}
		\end{align}
		for $\mu_{\rm max}$ a.e. $y$ with a maximizing ergodic measure $\mu_{\rm max}$
		since it follows from the pointwise ergodic theorem that
		\[
		\alpha(\phi)=\liminf_{n\to\infty}  \frac{S_n\phi(y)}{n}
		\]
		for $\mu_{\rm max}$ a.e. $y$, and we have
		\begin{align*}
		\alpha(\phi)=\liminf_{n\to\infty}  &\frac{S_n\phi(y)}{n}\le \sup_{x\in X} \liminf_{n\to\infty}  \frac{S_n\phi(x)}{n}\\
		&\le \sup_{x\in X} \limsup_{n\to\infty}  \frac{S_n\phi(x)}{n}
		\le\limsup_{n\to\infty}  \sup_{x\in X}  \frac{S_n\phi(x)}{n}\le \alpha(\phi)
		\end{align*}
		by \eqref{eqn:KB}.
		Using Lemma~\ref{lem:key}, we have
		\begin{align*}
		\alpha(\phi)=\sup_{x\in X} \liminf_{n\to \infty} \frac{S_n\phi(x)}{n}=\sup_{x\in X} \inf_{n\in\Nset} \frac{S_n\phi(x)}{n}.
		\end{align*}
		We remark that we can obtain Eq.~\eqref{eqn:EOminmax} from \eqref{eqn:KB}, \eqref{eqn:EO3}, \eqref{eqn:key} and the minimax inequality \eqref{eqn:minimax_ineq},
		in which case we need the pointwise ergodic theorem.
	\end{proof}

	\begin{proof}[Proof of Lemma~\ref{lem:key}]
		Let $f(x):=\inf_{n\in \Nset} \frac{S_n\phi(x)}{n}$ and $g(x):=\liminf_{n\to +\infty} \frac{S_n\phi(x)}{n}$.
		It is trivial that $\sup_{x\in X}f(x)\le\sup_{x\in X} g(x)$.
		Assume that
		\begin{align}\label{eqn:contradict}
			\sup_{x\in X}f(x)<\sup_{x\in X} g(x).
		\end{align}
		Let $\gamma=\sup_{x\in X} g(x)$ and $G=g^{-1}(\gamma)$. Note that $G$ is not empty by \eqref{eqn:key2}. In addition, it is easy to see that $g(T(x))=g(x)$ for all $x\in X$, which implies that $G$ is $T$-invariant.
		By \eqref{eqn:contradict}, we have
		\begin{align}\label{eqn:less}
		f(x)<g(x)=\gamma,\quad x\in G.
		\end{align}
		This implies that for each $x\in G$ there exists $n^*\in \Nset$ such that
		\begin{align}\label{eqn:finite}
			\frac{S_{n^*}\phi(x)}{n^*}=f(x).
		\end{align}
		Fix $x\in X$ and let $n_0\in \Nset$ be the largest positive integer which satisfies \eqref{eqn:finite}.
		Note that if infinitely many $n_0\in \Nset$ satisfy \eqref{eqn:finite}
		then it holds that $g(x)=\liminf_{n\to \infty} \frac{S_n\phi(x)}{n} \le f(x)$,
		which is impossible by \eqref{eqn:less}.
		Then we have
		\[
			\frac{S_{n_0}\phi(x)}{n_0}<\frac{S_{n_0+m}\phi(x)}{n_0+m},\quad m\in\Nset.
		\]
		We compute
		\[
			n_0 S_{n_0}\phi(x)+m S_{n_0}\phi(x)< n_0 S_{n_0}\phi(x)+n_0 S_m \phi(T^{n_0}(x))
		\]
		and thus
		\[
			\frac{S_{n_0}\phi(x)}{n_0}<\frac{S_{m}\phi(T^{n_0}(x))}{m},\quad m\in\Nset.
		\]
		Since $x\in G$ and $G$ is $T$-invariant, $x_1:=T^{n_0}(x)$ also belongs to $G$.
		Therefore, we can repeat the same discussion and deduce that there exists a positive integer sequence $\{n_i\}_{i\in\Nset}$ such that
		\begin{align}\label{eqn:inc}
			\frac{S_{n_0}\phi(x)}{n_0}<\frac{S_{n_1}\phi(x_1)}{n_1}<\cdots<\frac{S_{n_i}\phi(x_i)}{n_i}<\cdots\quad (<\gamma)
		\end{align}
		with
		\begin{align}\label{eqn:f}
			\frac{S_{n_i}\phi(x_i)}{n_i}=f(x_i),
		\end{align}
		where $x_{i+1}=T^{n_i}(x_i)$ for $i\in\Nset_0=\Nset\cup\{0\}$ with $x_0:=x$.
		By \eqref{eqn:inc}, the sequence $\{(1/n_i)S_{n_i}(x_i)\}_{i\in\Nset_0}$ is strictly increasing and bounded above by $\gamma$,
		thus it has a supremum $\beta$ and
		\begin{align}\label{eqn:b_g}
			\beta\le \gamma
		\end{align} holds.
		Moreover, by $(1/n_i)S_{n_i}\phi(x_i)\le\beta$ for all $i\in\Nset_0$, we have
		\begin{align}\label{eqn:gamma}
			S_{n_i}\phi(x_i)\le n_i\beta,\quad i\in\Nset_0.
		\end{align}
		From \eqref{eqn:gamma}, for $k\in\Nset$, we have
		\[
		S_{m_k}\phi(x)=\sum_{i=0}^k S_{n_i}\phi(x_i)\le m_k\beta,
		\]
		i.e.,
		\[
			\frac{S_{m_k}\phi(x)}{m_k}\le \beta,
		\]
		where $m_k=\sum_{i=1}^k n_i$.
		Letting $k\to\infty$, we obtain
		\[
			\gamma=g(x)=\liminf_{n\to \infty} \frac{S_n\phi(x)}{n}\le \liminf_{k\to \infty}\frac{S_{m_k}\phi(x)}{m_k}\le \beta.
		\]
		Combining with \eqref{eqn:b_g}, we have $\beta=\gamma$.
		
		On the other hand, from \eqref{eqn:inc} and \eqref{eqn:f},
		we see that
		\[
			\sup_{m\in\Nset} f(T^m(x))\ge \beta=\gamma,
		\]
		which contradicts to \eqref{eqn:contradict}.
	\end{proof}

		\begin{remark}\label{rmk:usc}
		The function $\inf_{n\in\Nset} (1/n)S_n\phi(x)$ has a maximum point in $X$.
		In fact, by the continuity of $\frac{1}{n} S_n\phi$ for each $n\in\Nset$, we see that
		$\{x\in X; (1/n)S_n\phi(x)\ge c\}$ is closed for each $c\in\Rset$ and thus
		\[
		\left\{x\in X; \inf_{n\in\Nset} (1/n)S_n\phi(x) \ge c\right\}
		=\bigcap_{n\in\Nset} \left\{x\in X; (1/n)S_n\phi(x)\ge c\right\}
		\]
		is also closed, which yields the upper semicontinuity of $\inf_{n\in\Nset} (1/n)S_n\phi(x)$.
		Note that all the maximum point of $\inf_{n\in\Nset} (1/n)S_n\phi(x)$ attain
		the maximum of $\liminf_{n\to\infty} (1/n)S_n\phi(x)$
		since
		\[
		\inf_{n\in\Nset} (1/n)S_n\phi(x)\le \liminf_{n\to\infty} (1/n)S_n\phi(x) \le \alpha(\phi),\quad x\in X
		\]
		However, not all the maximum point of $\liminf_{n\to\infty} (1/n)S_n\phi(x)$
		do not take the maximum point of $\inf_{n\in\Nset} (1/n)S_n\phi(x)$ (see Example~\ref{exm:counter}).
	\end{remark}
	\begin{remark}\label{rmk:supsup}
		The maximum ergodic average $\alpha(\phi)$ for $\phi\in C(X)$ does not coincides
		with neither
		\[
		\sup_{x\in X} \sup_{n\in\Nset} \frac{S_n\phi(x)}{n}=\sup_{n\in \Nset} \sup_{x\in X} \frac{S_n\phi(x)}{n}
		\]
		nor
		\[
		\inf_{x\in X} \inf_{n\in\Nset} \frac{S_n\phi(x)}{n}=\inf_{n\in \Nset} \inf_{x\in X} \frac{S_n\phi(x)}{n}
		\]
		in general (see Example~\ref{exm:counter} below).
	\end{remark}
	\begin{example}\label{exm:counter}
		Consider a subshift $\Sigma_A\subset \Omega=\{0,1,a\}^{\Nset_0}$ with the transition matrix
		\begin{align*}
			A=\begin{pmatrix}
				0 & 1 & 0\\
				1 & 0 &  0\\
				0 & 1 &  0
				\end{pmatrix},
		\end{align*}
		where $a\neq 0,1$ is a real parameter.
		We denote the shift map on $\Sigma_A$ by $\sigma$.
		Then $\Sigma_A$ contains only three points, $x=(10)^\infty, (01)^\infty, a(10)^\infty$.
		Let $\phi$ be a locally constant function given by $\phi(x)=x_0$ for $x=x_0 x_1\ldots\in \Sigma_A$.
		It is clear that the maximum ergodic average is $\alpha(\phi)=1/2$.
		Now we compute $(1/n)S_n\phi(x)$ for each $x=(10)^\infty, (01)^\infty, a(10)^\infty$.
		We have
		\begin{align*}
			\frac{S_n\phi((01)^\infty)}{n}&=\left\{
\begin{array}{ll}
\frac{1}{2} & (n:{\rm even})\\
\frac{1}{2}-\frac{1}{2n}& (n:{\rm odd})
\end{array}
\right.
,\\ 
\frac{S_n\phi((10)^\infty)}{n}&=\left\{
\begin{array}{ll}
\frac{1}{2} & (n:{\rm even})\\
\frac{1}{2}+\frac{1}{2n} & (n:{\rm odd})
\end{array}
\right.
,
\\
\frac{S_n\phi(a(10)^\infty)}{n}&=\left\{
\begin{array}{ll}
\frac{1}{2}+\frac{a}{n}  & (n:{\rm even})\\
\frac{1}{2}+\frac{a}{n}-\frac{1}{2n} & (n:{\rm odd})
\end{array}
\right.
.
		\end{align*}
		Letting $n\to \infty$, it is easy to see that Proposition~\ref{prop:jenkinson} holds in this example.
		Fix $n\in\Nset$. We see that 
		\[
		\sup_{x\in \Sigma_A} \frac{S_n\phi(x)}{n}
			=\left\{
\begin{array}{ll}
\frac{1}{2}  & (a<0)\\
\frac{1}{2}+\frac{a}{n} & (a> 0)
\end{array}
\right.
		\]
		holds when $n$ is even, and that
		\[
		\sup_{x\in \Sigma_A} \frac{S_n\phi(x)}{n}
			=\left\{
\begin{array}{ll}
\frac{1}{2}+\frac{1}{2n}  & (a<1)\\
\frac{1}{2}+\frac{2a-1}{2n} & (a> 1)
\end{array}
\right.
		\]
		holds when $n$ is odd.
		Thus it holds that
		\[
			 \inf_{n\in \Nset}\sup_{x\in \Sigma_A} \frac{S_n\phi(x)}{n}=\frac{1}{2}=\alpha(\phi)
		\]
		for any $a\neq 0,1$. Moreover, we compute
		\begin{align*}
			\inf_{n\in \Nset} \frac{S_n\phi((01)^\infty)}{n}&=0,\quad
			\inf_{n\in \Nset} \frac{S_n\phi((10)^\infty)}{n}=\frac{1}{2},\\
			\inf_{n\in \Nset} \frac{S_n\phi(a(01)^\infty)}{n}
			&=\left\{
\begin{array}{ll}
\frac{1}{2}  & (a>\frac{1}{2})\\
\frac{1}{2}+\frac{2a-1}{2n} & (a\le \frac{1}{2})
\end{array}
\right.
		\end{align*}
		with $a\neq 0,1$. Therefore, we obtain
		\[
			\sup_{x\in \Sigma_A} \inf_{n\in \Nset} \frac{S_n\phi(x)}{n}=\frac{1}{2}=\alpha(\phi).
		\]
		Here we remark that
		\[
			\inf_{n\in \Nset} \frac{S_n\phi((01)^\infty)}{n}=0<\liminf_{n\to \infty} \frac{S_n\phi((01)^\infty)}{n}=\alpha(\phi)=\frac{1}{2}
		\]
		(see Remark~\ref{rmk:usc}).
		In addition, note that
		\[
		\sup_{x\in \Sigma_A} \sup_{n\in \Nset} \frac{S_n\phi(x)}{n}\ge \sup_{n\in \Nset} \frac{S_n\phi((10)^\infty)}{n}=1>\alpha(\phi)=\frac{1}{2}
		\]
		and
		\[
		\inf_{x\in \Sigma_A} \inf_{n\in \Nset} \frac{S_n\phi(x)}{n}\le \inf_{n\in \Nset} \frac{S_n\phi((01)^\infty)}{n}=0<\alpha(\phi)=\frac{1}{2}
		\]
		 (see Remark~\ref{rmk:supsup}).
	\end{example}

	The following another characterization of the maximum ergodic average $\alpha(\phi)$
	also can be obtained by the minimax approach.
	\begin{theorem}\label{th:dual} For each $\phi\in C(X, \R)$, 
	\begin{align*}%\label{eqn:dual}
	\alpha(\phi)=\inf_{\psi\in C(X;\R)} \sup_{x\in X}(\phi(x)+\psi(x)-\psi\circ T(x))
	\end{align*}
\end{theorem}
	\begin{proof}
			Let $\Z:=\{\xi\in C(X);\ \exists \psi\in C(X)\ s.t.\ \xi=\psi-\psi\circ T\}, \W:=\M(X)$,
			and $F(\xi,\mu)=\int (\phi+\xi)d\mu$.
			Then $F(\xi,\mu)$ is affine on $\Z$ and on $\W$, and continuous on $\Z\times \W$.
			We use the minimax theorem and compute 
			\begin{align*}
				%\min_{\mu\in \M(X)}& \sup_{\xi\in \Y} -\langle f+\xi, \mu\rangle
				\sup_{\mu\in \M(X)} \inf_{\xi \in \Z} \int (\phi+\xi)d\mu
				=\sup_{\mu\in \M(X)} \begin{cases}\int \phi\ d\mu & (\mu\in \M_T)\\
					-\infty & (\mu\notin \M_T)\end{cases}
				=\sup_{\mu\in \M_T(X)} \int \phi \ d\mu.
			\end{align*}
			and
			\begin{align*}
				%\sup_{\xi\in\Y} \min_{\mu\in\M(X)} \langle \phi+\xi, \mu\rangle
				\inf_{\psi\in C(X)} \sup_{\mu\in \M(X)} \int (\phi+\psi-\psi\circ T)d\mu
				=\inf_{\psi\in C(X)} \sup_{x\in X}  (\phi+\psi-\psi\circ T)(x),
			\end{align*}
			which yields the desired result.
			Note that $\mu$ is $T$-invariant if and only if $\int \psi\circ T d\mu= \int\psi\ d\mu$ for all $\psi\in C(X)$, which implies that if $\mu\notin\M_T(X)$ then there exists $\psi'\in C(X)$ such that $\int \psi'\circ Td\mu\neq \int \psi'\ d\mu$ and thus $\{a\psi'\}_{a\in\Rset}$ provides $\inf_{a\in\Rset} \int a(\psi'\circ T-\psi') d\mu=-\infty$.
	\end{proof}

	%\newpage

	\subsection{Abstract variational principle for generalized pressure functions}
	In this section, using the minimax approach, we give a simpler proof of the abstract variational principle for generalized pressure functions presented in \cite{B22, B23}.
	We first briefly recall their main result. See \cite{B22, B23} for its details and related topics.
	Let $(X,d)$ be a (not necessary compact) metric space and let $B$ stands for the $\sigma$-algebra of Borel subsets of $X$. Let $\B$ be a Banach space over $\Rset$
	given by either the set of $B$-measurable and bounded functions, or that of continuous ones, or else that of continuous ones with compact support. We endow $\B$ with the supremum norm $||\phi||_\infty=\sup_{x\in X} |\phi(x)|$. Let $K=\M_a(X)$ be the set of $B$-measurable, regular, and normalized  finitely additive set functions on $X$ with the total variation norm.
	Then, by Banach-Alaoglu theorem, $K$ is compact with respect to the weak $*$-topology.
	Note that the set of Borel $\sigma$-additive probability measures $\M(X)$ is a subset of $K(=\M_a(X))$.
	A function $\Gamma:\B\to\Rset$ is called a \textit{pressure function} if $\Gamma$ satisfies the following conditions for all $\phi,\psi\in\B, \ c\in\Rset,\ \lambda\in[0,1]$;
	\begin{itemize}
		\item[(C1)] Monotonicity: $\phi\le \psi \quad \Rightarrow \quad \Gamma(\phi)\le \Gamma(\psi)$
		\item[(C2)] Translation invariance: $\Gamma(\phi+c)=\Gamma(\phi)+c$
		\item[(C3)] Convexity: $\Gamma(\lambda\phi+(1-\lambda)\psi)\le \lambda\Gamma(\phi)+(1-\lambda)\Gamma(\psi)$.
	\end{itemize}
	
	%\begin{lemma}\label{lem:Lip}
	%	For a pressure function $\Gamma:\B\to \Rset$,
	%	\[
	%		\Gamma(\phi_1)-\Gamma(\phi_2)\le \sup_{x\in X} (\phi_1-\phi_2)(x),\quad \forall \phi_1,\phi_2\in\B.
	%	\]
	%\end{lemma}
	%\begin{proof} Taking arbitrary $\phi_1,\phi_2\in\B$, we have
	%	\begin{align*}
	%	\Gamma(\phi_1)&-\Gamma(\phi_2)%=\Gamma((\phi_1-\phi_2)+\phi_2)-\Gamma(\phi_2)
	%	\le \Gamma(\phi_2+\sup_{x\in X} (\phi_1-\phi_2)(x))-\Gamma(\phi_2)\\
	%	&=\Gamma(\phi_2)+\sup_{x\in X} (\phi_1-\phi_2)(x)-\Gamma(\phi_2)=\sup_{x\in X} (\phi_1-\phi_2)(x)
	%	\end{align*}
	%	from the properties (C1) and (C2).%, which also implies 1-Lipschitz continuity of $\Gamma$.
	%\end{proof}
	\begin{theorem}[\cite{B22}]\label{th:VP}
		Let $\Gamma:\B\to\Rset$ be a generalized pressure function and let
		$\A_\Gamma=\{\xi\in\B; \Gamma(-\xi)\le 0\}$ and
		$\h(\mu)= \inf_{\xi\in\A_\Gamma}\ \int \xi\ d\mu$.
		Then the following abstract variational principle holds:
		\begin{align}\label{eqn:VP1}
			\Gamma(\phi)=\sup_{\mu\in K} \left(\h(\mu)+\int \phi\ d\mu \right), \quad \forall\phi\in\B.
		\end{align}
		Moreover,
		\begin{align}\label{eqn:VP2}
		\h(\mu)=\inf_{\phi\in\B} \left(\Gamma(\phi)-\int \phi\ d\mu\right),\quad \forall\mu\in K.
		\end{align}
	\end{theorem}
	We now give a short proof of Theorem~\ref{th:VP} from the minimax theorem.
	\begin{proof}
		Let $\tilde{\h}(\mu)= \begin{cases} -\h(\mu) & (\mu\in K)\\
					+\infty & (\mu\notin K)\end{cases}$.
		Since the pointwise infimum of proper convex and lower semicontinuous functions is also proper convex and lower semicontinuous, so is $\tilde{\h}$.
		By Fenchel-Moreau theorem, we have $\tilde{\h}^{**}=\tilde{\h}$.
		Therefore, if we prove $\tilde{\h}^*=\Gamma(\phi)$, we will obtain $\tilde{\h}=\tilde{\h}^{**}=\Gamma^*$%=-\inf_{\phi\in\B} (\Gamma(\phi)-\int \phi\ d\mu)$
		, i.e., Eq.~\eqref{eqn:VP2}.
		
		We now prove $\tilde{\h}^*=\Gamma(\phi)$, i.e., Eq.~\eqref{eqn:VP1}.
		%\[
		%\Gamma(\phi)=\sup_{\mu\in \B^*} (\langle \phi,\mu\rangle-\tilde{\h}(\mu))
		%=\sup_{\mu\in K} (\h(\mu)+\int \phi d\mu).
		%\]
		As in the original proof in \cite{B22}, we easily check $\Gamma(\phi)\ge\sup_{\mu\in K} (\h(\mu)+\int \phi d\mu)$. In fact, since
		for any $\phi\in \B$ we have
		$\Gamma(-(\Gamma(\phi)-\phi))=\Gamma(\phi)-\Gamma(\phi)=0$,
		it holds that
		\begin{align*}
			\h(\mu)+\int \phi d\mu \le \int (\Gamma(\phi)-\phi)d\mu+\int \phi d\mu
			=\Gamma(\phi)\int 1\ d\mu%- \langle \phi,\mu\rangle+ \langle {\phi},\mu\rangle
			=\Gamma(\phi).
		\end{align*}
		
		On the other hand, by the minimax theorem with the convexity of $\A_\Gamma$,% we obtain
		\begin{align*}
			&\sup_{\mu\in K}\left(\h(\mu)+\int \phi \ d\mu\right)=\sup_{\mu\in K}\inf_{\xi\in\A_\Gamma}\left(\int \xi\ d\mu+\int \phi \ d\mu\right)\\
			&=\inf_{\xi\in\A_\Gamma}\sup_{\mu\in K}\left( \int \xi\ d\mu+\int \phi \ d\mu\right)
			=\inf_{\xi\in\A_\Gamma}\sup_{x\in X}(\xi+ \phi)(x)\ge \Gamma(\phi)
		\end{align*}
		holds. Here the last inequality follows from% Lemma~\ref{lem:Lip} and
		\begin{align*}
			\Gamma(\phi)%=\Gamma(\phi)-\Gamma(-\xi)+\Gamma(-\xi)
			\le\Gamma(\phi)-\Gamma(-\xi)
			\le \sup_{x\in X}\ (\phi+ \xi)(x),\quad \forall \xi\in\A_\Gamma.
		\end{align*}
		%for any $\xi\in\B$ with $\Gamma(-\xi)\le 0$. 
		Note that
		\begin{align*}
		\Gamma(\phi_1)&-\Gamma(\phi_2)%=\Gamma((\phi_1-\phi_2)+\phi_2)-\Gamma(\phi_2)
		\le \Gamma(\phi_2+\sup_{x\in X} (\phi_1-\phi_2)(x))-\Gamma(\phi_2)=\sup_{x\in X} (\phi_1-\phi_2)(x)
		\end{align*}
		holds for arbitrary $\phi_1,\phi_2\in\B$ by the properties (C1) and (C2).
	\end{proof}
	\begin{remark}\ 
		\begin{itemize}
		\item[(i)] Note that (C3) implies that the set $\A_\Gamma:=\{\xi\in\B; \Gamma(-\xi)\le 0\}$ is convex.
	In the original paper \cite{B22}, the convexity of $\Gamma$ is assumed but for the proof of Theorem~\ref{th:VP} we only need to assume conditions (C1), (C2) and the convexity of $\A_\Gamma$. As a result, we see that $\Gamma$ is convex by Eq.~\eqref{eqn:VP1}.
		\item[(ii)]As pointed out in \cite{B22}, if the pressure function $\Gamma$ satisfies the following cohomology invariance condition
	\begin{itemize}
		\item[(C4)] $\Gamma(\phi+\psi-\psi\circ T)=\Gamma(\phi)=\Gamma(\phi+\psi\circ T-\psi),\quad \forall \phi,\psi\in\B$,
	\end{itemize}
	then we can replace the set $K$ in Eq.~\eqref{eqn:VP1} with
	\[
	K_{inv}=\left\{\mu\in K; \int\phi\circ T\ d\mu=\int \phi \ d\mu\ \text{for all}\ \phi\in\B \right\}.
	\]
	\end{itemize}
	\end{remark}
	%In fact, letting $\mu_+,\mu,\mu_-$ be the maximizers of $\Gamma(\phi+\psi-\psi\circ T),\Gamma(\phi),\Gamma(\phi+\psi\circ T-\psi)$ in Eq~\eqref{eqn:VP1} respectively, we see that
	%\begin{align*}
	%	\h(\mu)&+\langle \phi,\mu\rangle=\Gamma(\phi)=\Gamma(\phi+\psi-\psi\circ T)\\
	%	&=\h(\mu_+)+\langle \phi+(\psi-\psi\circ T),\mu_+\rangle
	%	\ge \h(\mu)+\langle \phi,\mu\rangle+\langle (\psi-\psi\circ T),\mu\rangle
	%\end{align*}
	%and
	%\begin{align*}
	%	\h(\mu)&+\langle \phi,\mu\rangle=\Gamma(\phi)=\Gamma(\phi+\psi\circ T-\psi)\\
	%	&=\h(\mu_-)+\langle \phi+(\psi\circ T-\psi),\mu_-\rangle
	%	\ge \h(\mu)+\langle \phi,\mu\rangle+\langle (\psi\circ T-\psi),\mu\rangle
	%\end{align*}
	%which implies that
	%\[
	%	\langle (\psi\circ T-\psi),\mu\rangle=0,\quad \forall\psi\in\B,
	%\]
	%i.e., $\mu$ is $T$-invariant.

	%\subsection{Comment on the classical pressure function}
	%The classical pressure function $P_{top}(\phi)$ satisfies (C)
	%\begin{align}\label{eqn:top_press_entropy}
	%	P_{top}(0)=h_{top}.
	%\end{align}
	%Using (C1),(C2),(C4) and Eq.~\eqref{eqn:top_press_entropy}, we see that
	%\begin{align*}
	%	\sup_{\mu\in\M(X)}\ \inf_{\phi\in C(X)} (P_{top}(\phi)-\langle \phi,\mu\rangle)
	%	=\inf_{\phi\in C(X)}\ \sup_{\M}(P_{top}(\phi)-\langle \phi,\mu\rangle)\\
	%	=\inf_{\phi\in C(X)}\ (P(\phi)-\inf_{x\in X} \phi(x))
	%	=\inf_{\phi\in C(X)}\ P(\phi-\inf_{x\in X} \phi(x))\\
	%	=\inf_{\{\phi\in C(X); \phi\ge 0\}} P(\phi)=P_{top}(0)=h_{top}.
	%\end{align*}
	%Therefore, we see that
	%\[
	%h_{top}=\sup_{\mu\in\M_T(X)}\ \inf_{\phi\in C(X)} (P_{top}(\phi)-\langle \phi,\mu\rangle)
	%\]
\section{Alternative approach: Fenchel-Rockafellar duality}\label{sec:FR}
In the proofs of Theorems~\ref{th:dual}, \ref{th:VP}, we use the minimax theorem presented in Section~\ref{sec:pre}.
However, the following form of minimax theorem is sufficient for our setting.
\begin{lemma}\label{lem:minimax_alt}
	Let $E$ be a Banach space and $E^*$ be its topological dual.
	Let $S\subset E$ and $K\subset E^*$ be non-empty closed convex subsets.
	Let $A:K\to\Rset$ be a finite, upper semicontinuous, and concave function.
	Then
	\[
		\sup_{\mu\in K} \left(\inf_{\xi\in S} \big\{\langle \xi,\mu\rangle +A(\mu)\big\}\right)
		= \inf_{\xi\in S}\left( \sup_{\mu\in K}\big\{\langle \xi,\mu\rangle +A(\mu)\big\}\right),
	\]
	where $\langle\cdot , \cdot \rangle$ stands for a natural paring of $E$ and $E^*$.
\end{lemma}
In this section, we describe a connection between minimax theorem of the form Lemma~\ref{lem:minimax_alt} and the Fenchel-Rockafellar duality \cite{R66}.
\begin{theorem}[Fenchel-Rockafellar duality]
Let $E$ be a Banach space and denote by $E^*$ its topological dual.
Let $f,g:E\to \Rset\cup\{+\infty\}$ be proper convex functions.
If either $f$ or $g$ is continuous at some point where both functions are finite, then
\begin{align*}
	\inf_{e\in E}(f(e)+g(e))=\sup_{e^*\in E^*}(-f^*(e^*)-g^*(-e^*)).
\end{align*}
where $f^*(e^*)=\sup_{e\in E}\{\langle e,e^*\rangle-f(e)\}, \quad g^*(e^*)=\sup_{e\in E}\{\langle e,e^*\rangle-g(e)\}$.
\end{theorem}
Now we prove the minimax theorem of the form Lemma~\ref{lem:minimax_alt} by using of the Fenchel-Rockafellar duality.
\begin{proof}[Proof of Lemma~\ref{lem:minimax_alt}]
	Consider $\tilde{A}(\mu)=\begin{cases} -A(\mu) & (\mu\in K)\\
					+\infty & (\mu\notin K)\end{cases}$.
	By the assumption, $\tilde{A}$ is proper, convex and lower semicontinuous function on $E^*$.
	Therefore, by the Fenchel-Moreau theorem, we have $\tilde{A}^{**}=\tilde{A}$.
	Let $f(\xi)=\sup_{\mu\in K} \big\{\langle \xi,\mu\rangle +A(\mu)\big\}$ and
	$g(\xi)= \begin{cases} 0 & (\xi\in S)\\
					+\infty & (\xi\notin S)\end{cases}$.
	Then $f=\tilde{A}^*$ holds since
	\begin{align*}
		\tilde{A}^*(\xi)&=\sup_{\mu\in E^*}\{\langle \xi,\mu\rangle-\tilde{A}(\mu)\}
		=\sup_{\mu\in E^*} \begin{cases} \langle \xi,\mu\rangle+A(\mu) & (\mu\in K)\\
					-\infty & (\mu\notin K)\end{cases}
		%=\sup_{\mu\in K} \{\langle \xi,\mu\rangle+A(\mu)\}
		=f(\xi),
	\end{align*}
	and thus $f,g$ are proper, convex and lower semicontinuous.
	It is easy to see that $f$ is continuous on $E$.
	The convex conjugates of $f$ and $g$ are given by
	\begin{align*}
		f^*(\mu)&=\tilde{A}^{**}(\mu)=\tilde{A}(\mu)
	\end{align*}
	and
	\begin{align*}
		g^*(\mu)=\sup_{\xi\in E}\left(\langle\xi,\mu\rangle-\begin{cases} 0 & (\xi\in S)\\
					+\infty & (\xi\notin S)\end{cases}\ \ \right) =\sup_{\xi\in S} \langle\xi,\mu\rangle.
	\end{align*}
	We compute
	\begin{align*}
		\inf_{\xi\in E} (f(\xi)+g(\xi))=\inf_{\xi\in S} \sup_{\mu\in K} \big\{A(\mu)+\langle\xi,\mu\rangle\big\}
	\end{align*}
	and
	\begin{align*}
		\sup_{\mu\in E^*} (-f^*(\mu)-g^*(-\mu))
		&=\sup_{\mu\in E^*} \left(\begin{cases} A(\mu) & (\mu\in K)\\
					-\infty & (\mu\notin K)\end{cases}-\sup_{\xi\in S} \langle \xi,-\mu\rangle\right)\\
		&=\sup_{\mu\in K}\inf_{\xi\in S} \big\{A(\mu)+\langle\xi,\mu\rangle\big\},
	\end{align*}
	which yields Lemma~\ref{lem:minimax_alt} by the Fenchel-Rockafellar duality.
\end{proof}

%%%%%%%%%%%%%%%%%%%%%%%%%%%%%%%%%%%%%%%%%%%%%%%%
%\vspace*{33pt}

%\noindent
%\textbf{Acknowledgement.}~ 
%The author was partially supported by JSPS KAKENHI Grant Number 23K22409.

%%%%%%%%%%%%%%%%%%%%%%%%%%%%%%%%%%%%%%%%%%%%%%%%%%%%%%%%
% \begin{thebibliography}{}
% \bibitem[Bou01]{Bou01} LA CONDITION DE WALTERS
% \bibitem[Cha11]{Cha11} Zero-temperature limit of one-dimensional Gibbs states via renormalization: the case of locally constant potentials. 
% \bibitem[Jen01]{Jen01} Rotation, entropy, and equilibrium states
% \bibitem[Jen06]{Jen06} Every ergodic measure is uniquely maximizing
% \bibitem[GarLop2018]{GarLop2018} Functions for relative maximization
% \bibitem[DGS]{DGS} Ergodic theory on compact spaces
% \bibitem[Mor2010]{Mor2010} Ergodic optimization for generic continuous functions
% \bibitem[Lep14]{Lep14} Zero temperature and selection of maximizing measure, http://www.math.univ-brest.fr/perso/renaud.leplaideur/KATHMANDU.pdf
% \end{thebibliography}

\bibliographystyle{alpha}
\bibliography{minmaxref}

\begin{bibdiv}
\begin{biblist}

\bib{B23}{article}{
      author={Bi\'{s}, Andrzej},
      author={Carvalho, Maria},
      author={Mendes, Miguel},
      author={Varandas, Paulo},
      author={Zhong, Xingfu},
       title={Correction: a convex analysis approach to entropy functions,
  variational principles and equilibrium states},
        date={2023},
     journal={Comm. Math. Phys.},
      volume={401},
      number={3},
       pages={3335\ndash 3342},
}

\bib{B22}{article}{
      author={Bi\'{s}, Andrzej},
      author={Carvalho, Maria},
      author={Mendes, Miguel},
      author={Varandas, Paulo},
       title={A convex analysis approach to entropy functions, variational
  principles and equilibrium states},
        date={2022},
        ISSN={0010-3616,1432-0916},
     journal={Comm. Math. Phys.},
      volume={394},
      number={1},
       pages={215\ndash 256},
}

\bib{Fan53}{article}{
      author={Fan, Ky},
       title={Minimax theorems},
        date={1953},
        ISSN={0027-8424},
     journal={Proc. Nat. Acad. Sci. U.S.A.},
      volume={39},
       pages={42\ndash 47},
}

\bib{Jen06a}{article}{
      author={Jenkinson, Oliver},
       title={Ergodic optimization},
        date={2006},
        ISSN={1078-0947},
     journal={Discrete Contin. Dyn. Syst.},
      volume={15},
      number={1},
       pages={197\ndash 224},
}

\bib{Jen17}{article}{
      author={Jenkinson, Oliver},
       title={Ergodic optimization in dynamical systems},
        date={2019},
        ISSN={0143-3857},
     journal={Ergodic Theory Dyn. Syst.},
      volume={39},
      number={10},
       pages={2593\ndash 2618},
}

\bib{R66}{article}{
      author={Rockafellar, R.~T.},
       title={Extension of {F}enchel's duality theorem for convex functions},
        date={1966},
        ISSN={0012-7094,1547-7398},
     journal={Duke Math. J.},
      volume={33},
       pages={81\ndash 89},
}

\bib{T20}{article}{
      author={Tsukamoto, Masaki},
       title={Double variational principle for mean dimension with potential},
        date={2020},
        ISSN={0001-8708,1090-2082},
     journal={Adv. Math.},
      volume={361},
       pages={106935, 53},
}

\end{biblist}
\end{bibdiv}

\end{document}